\documentclass[12pt]{amsart}
\usepackage{amssymb,latexsym,amsmath,amscd,amsthm,graphicx, color}
\usepackage[all]{xy}
\usepackage{pgf,tikz}
\usepackage{mathrsfs}
\usetikzlibrary{arrows}
\definecolor{uuuuuu}{rgb}{0.26666666666666666,0.26666666666666666,0.26666666666666666}
\definecolor{xdxdff}{rgb}{0.49019607843137253,0.49019607843137253,1.}
\definecolor{ffqqqq}{rgb}{1.,0.,0.}

\definecolor{uuuuuu}{rgb}{0.26666666666666666,0.26666666666666666,0.26666666666666666}
\definecolor{qqwuqq}{rgb}{0.,0.39215686274509803,0.}
\definecolor{zzttqq}{rgb}{0.6,0.2,0.}
\definecolor{xdxdff}{rgb}{0.49019607843137253,0.49019607843137253,1.}
\definecolor{qqqqff}{rgb}{0.,0.,1.}
\definecolor{cqcqcq}{rgb}{0.7529411764705882,0.7529411764705882,0.7529411764705882}

\theoremstyle{plain}

\theoremstyle{definition}
\newtheorem{prop}[subsection]{Proposition}

\newtheorem{example}[subsection]{Example}

\newtheorem{remark}[subsection]{Remark}

\newcommand{\uu}{\cup}

\newcommand{\sci}{\subset}
\newcommand{\set}[1]{\{#1\}}

\newcommand{\ga}{\alpha}
\newcommand{\gb}{\beta}

\newcommand{\tit}{\textit}

\newcommand{\D}[1]{\mathbb{#1}}
\newcommand{\F}[1]{\mathfrak{#1}}

\newcommand{\te}{\text}

\newcommand{\ol}{\overline}

\newcommand{\tri}{\triangle}
\begin{document}

To appear, Journal of Interdisciplinary Mathematics
\title{Center of mass and the optimal quantizers for some continuous and discrete uniform distributions}

\author{ Mrinal Kanti Roychowdhury}
\address{School of Mathematical and Statistical Sciences\\
University of Texas Rio Grande Valley\\
1201 West University Drive\\
Edinburg, TX 78539-2999, USA.}
\email{mrinal.roychowdhury@utrgv.edu}

\subjclass[2010]{60Exx, 62Exx, 94A34.}
\keywords{Center of mass, uniform distribution, optimal sets.}

\date{}
\maketitle

\pagestyle{myheadings}\markboth{Mrinal Kanti Roychowdhury}{Center of mass and the optimal quantizers for some continuous and discrete uniform distributions}

\begin{abstract}
In this paper, we first consider a flat plate (called a lamina) with uniform density $\rho$ that occupies a region $\mathfrak R$ of the plane. We show that the location of the center of mass, also known as the centroid, of the region equals the expected vector of a bivariate continuous random variable with a uniform probability distribution taking values on the region $\mathfrak R$. Using this property, we prove that the Voronoi regions of an optimal set of two-means with respect to the uniform distribution defined on a disc partition the disc into two regions bounded by the semicircles. Besides, we show that if an isosceles triangle is partitioned into an isosceles triangle and an isosceles trapezoid in the Golden ratio, then their centers of mass form a centroidal Voronoi tessellation of the triangle. In addition, using the properties of center of mass we determine the optimal sets of two-means and the corresponding quantization error for a uniform distribution defined on a region with uniform density bounded by a rhombus. Further, we determine the optimal sets of $n$-means, and the $n$th quantization errors for two different discrete uniform distributions for some positive integers $n\leq \text{card(supp}(P))$.
\end{abstract}

\section{Introduction}

Let us consider a flat plate, called a lamina, with uniform density $\rho$ that occupies a region $\F R$ of the plane. By the density $\rho$, it is meant that the mass per unit area of the region $\F R$ is $\rho$. The \tit{center of mass} or the \tit{centroid} of the region is the point in which the region will be perfectly balanced horizontally if suspended from the point. Let the region $\F R$ lies between the two curves $x_2=f(x_1)$ and $x_2=g(x_1)$ bounded by the lines $x_1=a$ and $x_1=b$, where $f(x_1)\geq g(x_1)$ for all $(x_1, x_2) \in \F R$. Let $A$ be the total area of the region $\F R$. Then, $A=\int_a^b (f(x_1)-g(x_1)) dx_1$. It is known that if $(\ol x_1, \ol x_2)$ is the centroid of the region $\F R$, then
 \begin{equation} \label{eq0} \ol x_1 =\frac 1 A \int_a^b x_1(f(x_1)-g(x_1)) dx_1 \te{ and } \ol x_2 =\frac 1 A \int_a^b \frac 12  ([f(x_1)]^2-[g(x_1)]^2) dx_1.
 \end{equation}
Given a finite subset $\ga$ of $\D R^2$, the Voronoi region generated by $a\in \ga$ is the set of all elements in $\D R^2$ which are nearest $a$, and is denoted by $M(a|\ga)$, i.e.,
\[M(a|\ga)=\set{x \in \D R^2 : \|x-a\|=\min_{b \in \ga}\|x-b\|},\]
where  $\|\cdot\|$ denotes the Euclidean norm on $\D R^2$. The set $\set{M(a|\ga) : a \in \ga}$ is called the \tit{Voronoi diagram} or \tit{Voronoi tessellation} of $\D R^2$ with respect to the set $\ga$. A Voronoi tessellation is called a \tit{centroidal Voronoi tessellation} (CVT) if each of the generators of the tessellation is also the centroid of its own Voronoi region. Centroidal Voronoi tessellations (CVTs) have become a useful tool in many applications ranging from geometric modeling, image and data analysis, and numerical partial differential equations, to problems in physics, astrophysics, chemistry, and biology (see \cite{DFG} for some more details).

Let $P$ denote a Borel probability measure on $\D R^2$.
For a finite set $\ga\sci\D R^2$, the error $\int \min_{a \in \ga} \|x-a\|^2 dP(x)$ is often referred to as the \tit{cost} or \tit{distortion error} for $\ga$, and is denoted by $V(P; \ga)$. For any positive integer $n$, write $V_n:=V_n(P)=\inf\set{V(P; \ga) :\alpha \subset \mathbb R^2, \text{ card}(\alpha) \leq n}$. Then, $V_n$ is called the $n$th quantization error for $P$. If $\int \| x\|^2 dP(x)<\infty$, then there is some set $\ga$ for
which the infimum is achieved (see \cite{ GKL, GL, GL1}). Such a set $\ga$ for which the infimum occurs and contains no more than $n$ points is called an \tit{optimal set of $n$-means}. The elements of an optimal set are called \tit{optimal centers}, or \te{optimal quantizers}. In some literature it is also referred to as \tit{principal points} (see \cite{MKT}, and the references therein). Let $\ga$ be an optimal set of $n$-means for a Borel probability measure $P$ on $\D R^2$. Let $a \in \ga$, and $M (a|\ga)$ be the Voronoi region generated by $a\in \ga$. Then, for every $a \in \ga $ it is well-known that $a=E(X : X \in M(a|\ga))$ (see \cite[Section~4.1]{GL1} and \cite[Chapter~6 and Chapter~11]{GG}). It has broad applications in signal processing and data compression. For some details and comprehensive lists of references one can see \cite{GG, GKL, GN, Z}. Rigorous mathematical treatment of the quantization theory is given in Graf-Luschgy's book (see \cite{GL1}). For some recent work in this direction one can see \cite{DR, R, RR}.

In this paper, we consider a flat plate (called a lamina) with uniform density $\rho$ that occupies a region $\F R$ of the plane. In Proposition~\ref{prop1}, we show that the location of the center of mass of the region equals the expected vector of a bivariate continuous random variable with a uniform probability distribution taking values on the region $\F R$. In other words, we show that with respect to the uniform distribution, the point in an optimal set of one-mean coincides with the center of mass of the lamina. If the probability distribution is not uniform, then Proposition~\ref{prop1} is not true. In this regard we give  a counter example Example~\ref{ex2}. In \cite{R}, Roychowdhury gave a conjecture that with respect to the uniform distribution defined on a disc, the Voronoi regions of the points in an optimal set of two-means partition the disc into two semicircles. Here by the semicircle it is meant one half of the disc bounded by the semicircle. Using Proposition~\ref{prop1}, in Proposition~\ref{prop2}, we prove that the conjecture is true. Besides, in Proposition~\ref{prop3}, we show that if an isosceles triangle is partitioned into an isosceles triangle and an isosceles trapezoid in the Golden ratio, then their centers of mass form a centroidal Voronoi tessellation of the triangle. In addition, in Proposition~\ref{prop4}, using the properties of center of mass, we determine the optimal set of two-means and the corresponding quantization error for a uniform distribution defined on a region with uniform density bounded by a rhombus. The proof of this proposition shows that the optimal set of two-means forms a centroidal Voronoi tessellation, but the converse is not true (see Remark~\ref{rem1}). Finally, in the last section, for two different discrete distributions $P$, we determine the optimal sets of $n$-means and the $n$th quantization errors for some positive integers $n\leq \te{card(supp}(P))$.

\section{Main Result}

For a bivariate continuous random variable $X:=(X_1, X_2)$ taking values on a region with some probability distribution, let $E(X)$ represent the expected vector of $X$. On the other hand, by $E(X_1)$ and $E(X_2)$, we denote the expectations of the random variables $X_1$ and $X_2$ with respect to their marginal distributions. By the position vector $\tilde a$ of a point $A$, it is meant that $\overrightarrow{OA}=\tilde a$. In the sequel, we will identify the position vector of a point $(a_1, a_2)$ by $(a_1, a_2):=a_1 i +a_2 j$, and apologize for any abuse in notation. Here $i$ and $j$ are the two unit vectors in the positive directions of $x_1$- and $x_2$-axes, respectively. For any two vectors $\vec u$ and $\vec v$, let $\vec u \cdot \vec v$ denote the dot product between the two vectors $\vec u$ and $\vec v$. Then, for any vector $\vec v$, by $(\vec v)^2$, we mean $(\vec v)^2:= \vec v\cdot \vec v$. Thus, $|\vec v|:=\sqrt{\vec v\cdot \vec v}$, which is called the length of the vector $\vec v$. For any two position vectors $\tilde a:=( a_1, a_2)$ and $\tilde b:=( b_1, b_2)$, we write $\rho(\tilde a, \tilde b):=(( a_1-b_1, a_2-b_2))^2=(a_1-b_1)^2 +(a_2-b_2)^2$ to represent the squared Euclidean distance between the two points $(a_1, a_2)$ and $(b_1, b_2)$.

 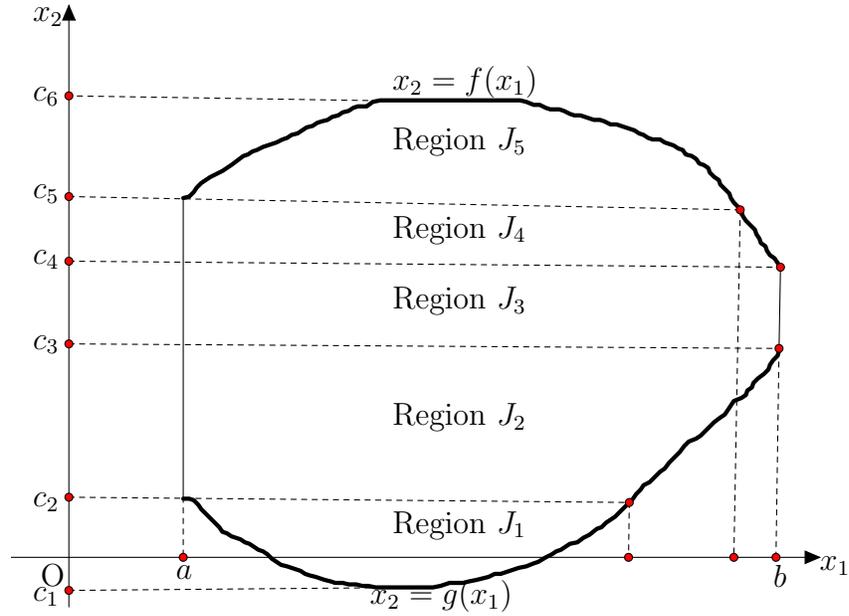
\begin{figure}
 \begin{tikzpicture}[line cap=round,line join=round,>=triangle 45,x=1.0cm,y=1.0cm]
\clip(-0.76,-0.8) rectangle (10.37,7.5);
\draw [line width=1.6pt,] (1.52,0.78)-- (1.6,0.78)-- (1.66,0.76)-- (1.72,0.7)-- (1.76,0.64)-- (1.82,0.58)-- (1.88,0.52)-- (1.94,0.46)-- (2.,0.4)-- (2.04,0.34)-- (2.1,0.32)-- (2.16,0.28)-- (2.22,0.26)-- (2.28,0.24)-- (2.34,0.2)-- (2.4,0.18)-- (2.46,0.14)-- (2.52,0.1)-- (2.58,0.08)-- (2.64,0.04)-- (2.7,-0.02)-- (2.76,-0.06)-- (2.82,-0.1)-- (2.9,-0.12)-- (2.96,-0.14)-- (3.02,-0.16)-- (3.08,-0.18)-- (3.14,-0.22)-- (3.2,-0.24)-- (3.28,-0.26)-- (3.34,-0.28)-- (3.4,-0.3)-- (3.48,-0.32)-- (3.54,-0.34)-- (3.64,-0.34)-- (3.72,-0.36)-- (3.82,-0.36)-- (3.9,-0.38)-- (3.98,-0.38)-- (4.08,-0.4)-- (4.18,-0.4)-- (4.26,-0.4)-- (4.36,-0.4)-- (4.46,-0.4)-- (4.56,-0.4)-- (4.64,-0.4)-- (4.72,-0.4)-- (4.84,-0.4)-- (4.94,-0.38)-- (5.,-0.36)-- (5.06,-0.34)-- (5.14,-0.34)-- (5.24,-0.34)-- (5.32,-0.32)-- (5.42,-0.3)-- (5.5,-0.28)-- (5.58,-0.26)-- (5.66,-0.22)-- (5.74,-0.22)-- (5.8,-0.2)-- (5.86,-0.18)-- (5.92,-0.16)-- (5.98,-0.14)-- (6.08,-0.12)-- (6.18,-0.08)-- (6.24,-0.06)-- (6.34,0.)-- (6.42,0.04)-- (6.48,0.08)-- (6.54,0.1)-- (6.6,0.14)-- (6.66,0.18)-- (6.76,0.2)-- (6.84,0.26)-- (6.92,0.3)-- (7.02,0.38)-- (7.08,0.4)-- (7.14,0.46)-- (7.2,0.5)-- (7.26,0.54)-- (7.32,0.6)-- (7.38,0.66)-- (7.44,0.72)-- (7.5,0.78)-- (7.54,0.84)-- (7.6,0.9)-- (7.64,0.96)-- (7.72,1.)-- (7.8,1.08)-- (7.86,1.14)-- (7.92,1.18)-- (7.98,1.26)-- (8.06,1.34)-- (8.12,1.4)-- (8.18,1.48)-- (8.24,1.54)-- (8.34,1.58)-- (8.4,1.62)-- (8.46,1.68)-- (8.52,1.74)-- (8.58,1.8)-- (8.66,1.86)-- (8.72,1.9)-- (8.76,1.96)-- (8.8,2.02)-- (8.84,2.08)-- (8.94,2.12)-- (9.,2.16)-- (9.02,2.22)-- (9.08,2.26)-- (9.1,2.32)-- (9.16,2.38)-- (9.24,2.44)-- (9.3,2.5)-- (9.32,2.56)-- (9.36,2.62)-- (9.42,2.68)-- (9.44,2.74);
\draw [line width=1.6pt,] (1.52,4.78)-- (1.6,4.8)-- (1.66,4.86)-- (1.7,4.92)-- (1.76,4.98)-- (1.88,5.06)-- (2.04,5.14)-- (2.1,5.18)-- (2.16,5.22)-- (2.24,5.28)-- (2.32,5.34)-- (2.38,5.36)-- (2.44,5.4)-- (2.5,5.42)-- (2.58,5.46)-- (2.64,5.48)-- (2.7,5.5)-- (2.76,5.52)-- (2.84,5.56)-- (2.9,5.58)-- (2.96,5.62)-- (3.04,5.64)-- (3.1,5.68)-- (3.16,5.72)-- (3.24,5.74)-- (3.3,5.78)-- (3.36,5.8)-- (3.44,5.82)-- (3.5,5.86)-- (3.58,5.88)-- (3.64,5.92)-- (3.72,5.94)-- (3.78,5.98)-- (3.84,6.)-- (3.96,6.)-- (4.04,6.06)-- (4.14,6.08)-- (4.24,6.08)-- (4.32,6.08)-- (4.42,6.08)-- (4.52,6.08)-- (4.6,6.08)-- (4.68,6.08)-- (4.78,6.08)-- (4.92,6.08)-- (5.02,6.08)-- (5.14,6.08)-- (5.22,6.08)-- (5.3,6.08)-- (5.4,6.08)-- (5.48,6.08)-- (5.58,6.08)-- (5.66,6.08)-- (5.74,6.08)-- (5.82,6.08)-- (5.9,6.08)-- (6.,6.08)-- (6.08,6.06)-- (6.18,6.02)-- (6.26,6.02)-- (6.34,5.98)-- (6.4,5.96)-- (6.52,5.96)-- (6.58,5.94)-- (6.66,5.9)-- (6.72,5.88)-- (6.8,5.88)-- (6.86,5.86)-- (6.92,5.84)-- (6.98,5.82)-- (7.06,5.82)-- (7.12,5.8)-- (7.2,5.76)-- (7.28,5.76)-- (7.34,5.74)-- (7.4,5.72)-- (7.46,5.7)-- (7.52,5.68)-- (7.58,5.64)-- (7.64,5.62)-- (7.7,5.58)-- (7.76,5.56)-- (7.86,5.56)-- (7.92,5.52)-- (7.98,5.5)-- (8.04,5.48)-- (8.1,5.44)-- (8.16,5.42)-- (8.2,5.36)-- (8.28,5.34)-- (8.34,5.3)-- (8.38,5.24)-- (8.44,5.22)-- (8.5,5.18)-- (8.52,5.12)-- (8.58,5.06)-- (8.64,5.02)-- (8.66,4.96)-- (8.72,4.9)-- (8.78,4.84)-- (8.8,4.78)-- (8.84,4.72)-- (8.9,4.66)-- (8.94,4.6)-- (8.96,4.54)-- (9.,4.48)-- (9.06,4.42)-- (9.12,4.36)-- (9.14,4.3)-- (9.18,4.24)-- (9.2,4.18)-- (9.26,4.14)-- (9.28,4.08)-- (9.32,4.02)-- (9.38,3.98)-- (9.42,3.92)-- (9.46,3.86);
\draw (9.46,3.86)-- (9.44,2.74);
\draw (1.52,4.78)-- (1.52,0.78);
\draw [->] (0.,-0.66) -- (0.,7.36);
\draw [->] (-1.,0.) -- (10.,0.);
\draw (-0.52,0.06) node[anchor=north west] {O};
\draw (-0.62,7.46) node[anchor=north west] {$x_2$};
\draw (9.84,0.16) node[anchor=north west] {$x_1$};
\draw (-0.62,1.0) node[anchor=north west] {$c_2$};
\draw (-0.62,3.1) node[anchor=north west] {$c_3$};
\draw [dash pattern=on 2pt off 2pt] (0.,-0.44)-- (4.68,-0.4);
\draw (-0.62,4.23) node[anchor=north west] {$c_4$};
\draw (-0.62,5.1) node[anchor=north west] {$c_5$};
\draw [dash pattern=on 2pt off 2pt] (4.116923076923074,6.075384615384605)-- (0.,6.14);
\draw (-0.62,6.40) node[anchor=north west] {$c_6$};
\draw (3.86,-0.16) node[anchor=north west] {$x_2=g(x_1)$};
\draw (4.16, 6.68) node[anchor=north west] {$x_2=f(x_1)$};
\draw (4.16,0.76) node[anchor=north west] {$\text{Region } J_1$};
\draw (4.16,2.20) node[anchor=north west] {$\text{Region } J_2$};
\draw (4.16,3.74) node[anchor=north west] {$\text{Region } J_3$};
\draw (4.16,4.69) node[anchor=north west] {$\text{Region } J_4$};
\draw (4.16,5.86) node[anchor=north west] {$\text{Region } J_5$};
\draw [dash pattern=on 2pt off 2pt] (0.,0.8)-- (7.45,0.73);
\draw [dash pattern=on 2pt off 2pt] (0.,2.84)-- (9.440701306981175,2.7792731909467645);
\draw [dash pattern=on 2pt off 2pt] (0.,3.94)-- (9.46,3.86);
\draw [dash pattern=on 2pt off 2pt] (0.,4.8)-- (8.92461538461536,4.623076923076903);
\draw (-0.62,-0.25) node[anchor=north west] {$c_1$};
\draw [dash pattern=on 2pt off 2pt] (1.52,0.8)-- (1.52,0.);
\draw (1.28,0.04) node[anchor=north west] {$a$};
\draw [dash pattern=on 2pt off 2pt] (9.4,0.)-- (9.440701306981175,2.7792731909467645);
\draw [dash pattern=on 2pt off 2pt] (8.92461538461536,4.623076923076903)-- (8.84,0.);
\draw [dash pattern=on 2pt off 2pt] (7.45,0.73)-- (7.44,0.);
\draw (9.22,0.04) node[anchor=north west] {$b$};
\begin{scriptsize}
\draw [fill=xdxdff] (0.,15.38) circle (2.5pt);
\draw [fill=qqqqff] (21.,0.) circle (2.5pt);
\draw [fill=ffqqqq] (0.,-0.44) circle (1.5pt);
\draw [fill=ffqqqq] (0.,6.14) circle (1.5pt);
\draw [fill=ffqqqq] (0.,4.8) circle (1.5pt);
\draw [fill=ffqqqq] (0.,3.94) circle (1.5pt);
\draw [fill=ffqqqq] (0.,2.84) circle (1.5pt);
\draw [fill=ffqqqq] (0.,0.8) circle (1.5pt);
\draw [fill=ffqqqq] (7.45,0.73) circle (1.5pt);
\draw [fill=ffqqqq] (9.440701306981175,2.7792731909467645) circle (1.5pt);
\draw [fill=ffqqqq] (9.46,3.86) circle (1.5pt);
\draw [fill=ffqqqq] (8.92461538461536,4.623076923076903) circle (1.5pt);
\draw [fill=ffqqqq] (1.52,0.) circle (1.5pt);
\draw [fill=ffqqqq] (9.4,0.) circle (1.5pt);
\draw [fill=ffqqqq] (8.84,0.) circle (1.5pt);
\draw [fill=ffqqqq] (7.44,0.) circle (1.5pt);
\end{scriptsize}
\end{tikzpicture}
\caption{Partition of the region $\F R$.} \label{Fig1}
\end{figure}

Let us now prove the following proposition.

\begin{prop} \label{prop1} Let $(\ol x_1, \ol x_2)$ be the center of mass of a lamina with uniform density $\rho$. Let $X:=(X_1, X_2)$ be a bivariate continuous random variable with uniform distribution taking values on the region $\F R$ occupied by the lamina. Then,
\[E(X)=(E(X_1), E(X_2))=(\ol x_1, \ol x_2).\]
\end{prop}
\begin{proof} Let $f(x_1, x_2)$ be the probability density function (pdf) of the bivariate continuous random variable $X:=(X_1, X_2)$ taking values on the region $\F R$ with respect to the uniform distribution. Let $A$ represent the area of the region. Then,
\[f(x_1, x_2)=\left\{\begin{array}{ccc}
\frac 1 A & \te{ for } (x_1, x_2) \in \F R, \\
\ 0  & \te{ otherwise}.
\end{array}\right.
\]
Let $f_1(x_1)$ and $f_2(x_2)$ represent the marginal pdfs of the random variables $X_1$ and $X_2$, respectively. Then, following the definitions in Probability Theory, we have
\[f_1(x_1) =\int_{g(x_1)}^{f(x_1)}f(x_1, x_2) dx_2  \te{ for } a\leq x_1\leq b.\]
To find $f_2(x_2)$ we have to proceed as follows: Split the region $\F R$ into five regions such as $J_1, J_2, \cdots,  J_5$ (see Figure~\ref{Fig1}). The regions $J_1, J_2, \cdots, J_5$ heavily depend on the two functions $f(x_1)$ and $g(x_1)$. We might have even more than five regions, or less in some cases. Thus, $J_1, J_2, \cdots, J_5$ are bounded by the lines $x_1=a$, $x_1=b$, and $x_2=c_i$ for $1\leq i \leq 6$, and the curves $x_2=g(x_1)$ and $x_2=f(x_1)$. Hence, as shown in Figure~\ref{Fig1}, we have
\begin{align*}
J_1&=\set{(x_1, x_2) : \min (g^{-1}(x_2)) \leq x_1\leq \max (g^{-1}(x_2)) \te{ and } c_1\leq x_2\leq c_2},\\
J_2&=\set{(x_1, x_2) : a\leq x_1\leq g^{-1}(x_2) \te{ and } c_2\leq x_2\leq c_3},\\
J_3&=\set{(x_1, x_2) : a\leq x_1\leq b \te{ and } c_3\leq x_2\leq c_4},\\
J_4&=\set{(x_1, x_2) : a\leq x_1\leq f^{-1}(x_2) \te{ and } c_4\leq x_2\leq c_5},\\
J_5&=\set{(x_1, x_2) : \min (f^{-1}(x_2)) \leq x_1\leq \max (f^{-1}(x_2))\te{ and } c_5\leq x_2\leq c_6},
\end{align*}
yielding
\begin{align*}
f_2(x_2)&= \left\{\begin{array} {ll }
\int_{ \min (g^{-1}(x_2))}^{\max (g^{-1}(x_2))}f(x_1, x_2) dx_1 & \te{ for }  c_1\leq x_2\leq c_2,\\
\int_a^{g^{-1}(x_2)} f(x_1, x_2) dx_1 & \te{ for } c_2\leq x_2\leq c_3, \\
\int_a^{b} f(x_1, x_2) dx_1 & \te{ for } c_3\leq x_2\leq c_4, \\
\int_a^{f^{-1}(x_2)} f(x_1, x_2) dx_1& \te{ for } c_4\leq x_2\leq c_5, \\
\int_{\min (f^{-1}(x_2))}^{\max (f^{-1}(x_2))} f(x_1, x_2) dx_1 & \te{ for } c_5\leq x_2\leq c_6.
\end{array} \right.
\end{align*}
Recall that for any $(x_1, x_2) \in \F R$, $f(x_1, x_2)=\frac 1 A$, and zero, otherwise. Thus, we have
\[E(X_1) =\int_a^b x_1 f_1(x_1)\, dx_1=\frac 1 A \int_a^b \int_{g(x_1)}^{f(x_1)}x_1 dx_2  dx_1=\frac 1 A \int_a^b x_1(f(x_1)-g(x_1))dx_1, \]
which by \eqref{eq0} implies that $E(X_1)=\ol x_1$. To show $E(X_2)=\ol x_2$, we will mainly use the changing in the order of integration in the regions of double integrals. We have
\begin{align*}
&E(X_2)=\int_{c_1}^{c_6} x_2f_2(x_2)dx_2\\
&=\frac 1 A\Big[\Big(\int_{c_1}^{c_2} \int_{ \min (g^{-1}(x_2))}^{\max (g^{-1}(x_2))}x_2dx_1dx_2+\int_{c_2}^{c_3} \int_a^{g^{-1}(x_2)}x_2 dx_1dx_2\Big) +\int_{c_3}^{c_4} \int_a^{b} x_2 dx_1dx_2\\
&\qquad +\Big(\int_{c_4}^{c_5}\int_a^{f^{-1}(x_2)} x_2 dx_1dx_2+\int_{c_5}^{c_6} \int_{\min (f^{-1}(x_2))}^{\max (f^{-1}(x_2))} x_2 dx_1dx_2\Big)\Big]\\
&=\frac 1 A\Big[\int_{a}^{b} \int_{g(x_1)}^{c_3} x_2dx_2dx_1+\int_{a}^{b} \int_{c_3}^{c_4}x_2 dx_2dx_1 +\int_{a}^{b} \int_{c_4}^{f(x_1)} x_2 dx_2dx_1\Big]\\
&=\frac 1 A \int_{a}^{b} \int_{g(x_1)}^{f(x_1)} x_2dx_2dx_1=\frac 1 A\int_a^b \frac 12  ([f(x_1)]^2-[g(x_1)]^2) dx_1,
\end{align*}
which by \eqref{eq0} implies that $E(X_2)=\ol x_2$. Hence,
\[E(X)=\frac 1 A \iint(x_1 i+x_2 j) dx_1dx_2 =i \int x_1 f_1(x_1) dx_1+j \int x_2 f_2(x_2) dx_2=\ol x_1 i +\ol x_2j=(\ol x_1, \ol x_2),\]
and thus, the proof of the proposition is complete.
\end{proof}

\begin{figure}
\begin{tikzpicture}[line cap=round,line join=round,>=triangle 45,x=0.5 cm,y=0.5 cm]
\clip(-1.468633853040174,-4.9542037219953405) rectangle (8.83137139104701,8.69023360057558);
\draw (4.,0.)-- (0.,-4.);
\draw (6.,0.)-- (-1.,0.);
\draw (3.7908462683403227,0.1260583149492831) node[anchor=north west] {$1$};
\draw (-.75498054366825326,4.5491024772215235) node[anchor=north west] {$1$};
\draw (-1.47498054366825326,-3.443333586768474) node[anchor=north west] {$-1$};
\draw [->] (-1.,0.) -- (6.,0.);
\draw (-.97498054366825326,-0.0000521803249616) node[anchor=north west] {O};
\draw [shift={(0.,0.)}] plot[domain=0.:1.5707963267948966,variable=\t]({1.*4.*cos(\t r)+0.*4.*sin(\t r)},{0.*4.*cos(\t r)+1.*4.*sin(\t r)});
\draw [->] (0.,-4.79150771869559) -- (0.,6.);
\draw [color=qqqqff](2.8111939029040953,3.97162413883067) node[anchor=north west] {$x_2=\sqrt{1-x_1^2}$};
\draw [color=qqqqff](1.702237331004131,-1.6176811164505027) node[anchor=north west] {$x_2=x_1-1$};
\end{tikzpicture}
\caption{The bounded region described in Example~\ref{example1}} \label{Fig2}
\end{figure}
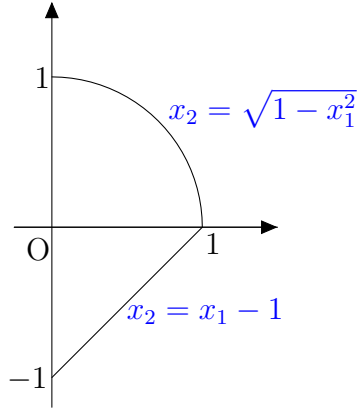

In support of the proposition, we give the following example.
\begin{example} \label{example1} Let us consider a lamina with uniform density $\rho$ which occupies a region $\F R$ in the plane bounded by the circle $x_1^2+x_2^2=1$, and the lines $x_2=x_1-1$ and $x_1=0$ (see Figure~\ref{Fig2}). Let $A$ be the area of the region. Then, $A=\frac \pi 4+\frac 12=\frac {2+\pi}{4}$. Let $(\ol x_1, \ol x_2)$ be the centroid of the region $\F R$.
 Here $f(x_1)=\sqrt{1-x_1^2}$ and $g(x_1)=x_1-1$. Then, using the formulas given by \eqref{eq0}, we have
\begin{align*} \ol x_1&=\frac 1A \int_{0}^1 x_1(f(x_1)-g(x_1))dx_1=\frac{2}{2+\pi }, \te{ and }\\
\ol x_2& =\frac 1 A\int_{0}^1 \frac 12([f(x_1)]^2-[g(x_1)]^2) dx_1=\frac{2}{3 (2+\pi )}.
\end{align*}
Let $f(x_1, x_2)$ be the pdf of a bivariate continuous random variable $X:=(X_1, X_2)$ with uniform distribution taking values on $\F R$.  Then,
$f(x_1, x_2)=\frac 1 A$ for $(x_1, x_2) \in \F R$, and $f(x_1, x_2)=0$ if $(x_1, x_2) \not \in \F R$, where
\[\F R=\set{(x_1, x_2) : -1\leq x_2\leq 0, \, 0\leq x_1\leq g^{-1}(x_2)}\uu \set{(x_1, x_2) : 0\leq x_2\leq 1, \, 0\leq x_1\leq f^{-1}(x_2)}.\]
Let $f_1(x_1)$ and $f_2(x_2)$ be the marginal distributions of $X_1$ and $X_2$, respectively. Then,
\[f_1(x_1)=\int_{g(x_1)}^{f(x_1)} f(x_1, x_2) dx_2=\frac 1 A (f(x_1)-g(x_1))=\frac 1 A (\sqrt{1-x^2}-(x-1)) \te{ for } 0\leq x_1\leq 1,\]
and
\[f_2(x_2)=\left\{\begin{array} {ll}
\int_0^{g^{-1}(x_2)} f(x_1, x_2) dx_2=\frac 1 A (x_2+1) & \te{ for } -1\leq x_2\leq 0,\\
\int_0^{f^{-1}(x_2)} f(x_1, x_2) dx_2=\frac 1 A \sqrt{1-x_2^2} & \te{ for } 0\leq x_2\leq 1.
\end{array} \right.\]
Thus,
\[E(X_1)=\frac 1 A \int_0^1 x_1(\sqrt{1-x^2}-(x-1))\,dx_1=\frac{2}{2+\pi}, \te{ and } \]
\[E(X_2)=\frac 1 A\Big(\int_{-1}^0 x_2(x_2+1) dx_2+\int_0^1 x_2\sqrt{1-x_2^2} \,dx_2\Big)=\frac{2}{3 (2+\pi )}, \]
implying $E(X_1)=\ol x_1$ and $E(X_2)=\ol x_2$.
\end{example}

If the bivariate continuous random variable $X:=(X_1, X_2)$ is not uniformly distributed on the region $\F R$, then the Proposition~\ref{prop1} is not true. In this regard, we give the following counter example.

\begin{example} \label{ex2}
Let $\F R$ be the square with vertices $O(0, 0)$, $A(1,0)$, $B(1, 1)$, and $C(0, 1)$ occupied by a lamina with uniform density $\rho$. Then, its area $A$ is given by $A=1$. Let $(\ol x_1, \ol x_2)$ be its center of mass. Then, we see that
$(\ol x_1, \ol x_2)=(\frac 12, \frac 12).$
Let $X:=(X_1, X_2)$ be a bivariate continuous random variable with probability density function $f(x_1, x_2)$ taking values on the square $\F R$ given by
\[f(x_1, x_2)=\left\{\begin{array} {cc}
4 x_1x_2 & \te{ if } (x_1, x_2) \in \F R, \\
0 & \te{ otherwise}.
\end{array}\right. \]
Then, if $f_1(x_1)$ and $f_2(x_2)$ are marginal pdfs of $X_1$ and $X_2$, respectively, we have
\[f_1(x_1)=\left\{\begin{array}{ll }
2x_1 & \te{ for } 0<x_1<1,\\
0 & \te{ otherwise},
\end{array}
\right.
\te{ and } f_2(x_2)=\left\{\begin{array}{ll }
2x_2 & \te{ for } 0<x_2<1,\\
0 & \te{ otherwise}.
\end{array}
\right.
\]
Thus, we see that $E(X_1)=\int_0^1 x_1f(x_1) dx_1=\frac 23$, and $E(X_2)=\int_{0}^1 x_2f_2(x_2)dx_2=\frac 23$, implying
\[E(X)=\frac 1 A \iint(x_1 i+x_2 j) dx_1dx_2 =i \int x_1 f_1(x_1) dx_1+j \int x_2 f_2(x_2) dx_2=\frac 23  i +\frac 23j=(\frac 23, \frac 23),\]
i.e., $E(X):=(E(X_1), E(X_2))=(\frac 23, \frac 23)\neq (\ol x_1, \ol x_2)$.
\end{example}

In the following proposition we use Proposition~\ref{prop1} and prove a conjecture given by Roychowdhury (see \cite[Conjecture~2.7]{R}).

\begin{prop} \label{prop2}
The Voronoi regions of the points in an optimal set of two-means with respect to the uniform distribution defined on a disc partition the disc into two regions bounded by the semicircles.
\end{prop}
\begin{proof}
It is enough to prove the proposition for the disc bounded by the circle given by the equation $x_1^2+x_2^2=1$. Let $P$ and $Q$ be the two points in an optimal set of two-means with respect to the uniform distribution on the disc. Let $\rho$ be the density, i.e., mass per unit area of the disc. Due to rotational symmetry of the disc about its center, without any loss of generality, we can assume that the boundary of the Voronoi regions of the points $P$ and $Q$ cut the circle at the points $A$ and $B$, and the line $AB$ is parallel to the $x_1$-axis. Thus, we can take the coordinates of $A$ and $B$ as $(-a, b)$ and $(a, b)$, respectively, where $a^2+b^2=1$. Notice that due to rotational symmetry of the disc we can assume that $b\geq 0$. Moreover, we can assume that $P$ is above the line $AB$, and $Q$ is below the line $AB$.  Due to Proposition~\ref{prop1}, we use the formulas given by \eqref{eq0} to calculate the locations of $P$ and $Q$. Let $(u_1, u_2)$ and $(v_1, v_2)$ be the coordinates of $P$ and $Q$, respectively. Here there is no need to use the formula to calculate $u_1$ and $v_1$ because, by the symmetry principle, the center of mass must lie on the $x_2$-axis, so $u_1=v_1=0$. Now, we calculate $u_2$ and $v_2$ as follows:
\[u_2=\frac{  \int_{-a}^a \frac{1}{2} \left(\left(1-x_1^2\right)-\left(1-a^2\right)\right) \, dx_1}{  \int_{-a}^a \left(\sqrt{1-x_1^2}-\sqrt{1-a^2}\right) \, dx_1}=\frac{2 a^3}{3 \left(\sin ^{-1}(a)-a \sqrt{1-a^2}\right)},\] and
\begin{align*} v_2&=\frac{  0+    \int_{-a}^a \frac{1}{2} \left(\left(1-a^2\right)-\left(1-x_1^2\right)\right) \, dx_1+  0}{   \int_{-1}^{-a} 2 \sqrt{1-x_1^2} \, dx_1+   \int_{-a}^a \left(\sqrt{1-a^2}+\sqrt{1-x_1^2}\right) \, dx_1+   \int_a^1 2 \sqrt{1-x_1^2} \, dx_1}\\
&=-\frac{2 a^3}{3 \left(a \sqrt{1-a^2}+\sin ^{-1}(a)+2 \cos ^{-1}(a)\right)}.
\end{align*}
Since the boundary of the Voronoi regions of any two points in an optimal set is the perpendicular bisector of the line segment joining the two points, we can say that the point $(\frac 12(u_1+v_1), \frac 12(u_2+v_2))$ lies on the line $AB$ yielding
\[\frac 12(u_2+v_2)=b=\sqrt{1-a^2}.\]  Putting the values of $u_2$ and $v_2$ in the above equation, and solving it, we have $a=1$, and so $b=0$ implying the fact that the line $AB$ coincides with the $x_1$-axis. In other words, the boundary of the Voronoi regions of the two points $P$ and $Q$ coincides with a diagonal of the circle. Thus, the proof of the proposition is complete.
\end{proof}

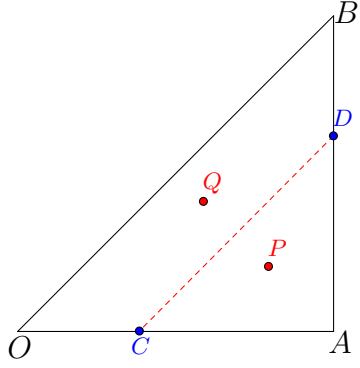
\begin{figure}
\begin{tikzpicture}[line cap=round,line join=round,>=triangle 45,x=0.7 cm,y=0.7 cm]
\clip(-0.256773307992824,-0.521997085840988) rectangle (6.891616935909634,6.459113250976666);
\draw (0.,0.)-- (6.,0.);
\draw (6.,0.)-- (6.,6.);
\draw (0.,0.)-- (6.,6.);
\draw (-.40104318445781938,0.11768049836342453) node[anchor=north west] {$O$};
\draw (5.696146552634362,0.21768049836342453) node[anchor=north west] {$A$};
\draw (5.796146552634362,6.47539013831697) node[anchor=north west] {$B$};
\draw [dash pattern=on 2pt off 2pt,color=ffqqqq] (5.997210431844581,3.717891415901172)-- (2.3118399324252987,0.007786886284447148);
\begin{scriptsize}
\draw [fill=ffqqqq] (4.76393,1.23607) circle (1.5pt);
\draw[color=ffqqqq] (4.9336471333544525,1.5907648189209163) node {$P$};
\draw [fill=ffqqqq] (3.52786,2.47214) circle (1.5pt);
\draw[color=ffqqqq] (3.6969456234822102,2.8274663287931583) node {$Q$};
\draw [fill=qqqqff] (2.3118399324252987,0.007786886284447148) circle (1.5pt);
\draw[color=qqqqff] (2.3365739626227437,-0.26535132509766667) node {$C$};
\draw [fill=qqqqff] (5.997210431844581,3.717891415901172) circle (1.5pt);
\draw[color=qqqqff] (6.170348643226695,4.0641678386654) node {$D$};
\end{scriptsize}
\end{tikzpicture}
\caption{$P$ and $Q$ form a centroidal Voronoi tessellation.}\label{Fig3}
\end{figure}

In the sequel by the triangle it it meant the lamina bounded by the triangle, and by the isosceles trapezoid it is meant the lamina bounded by the isosceles trapezoid. We now state and prove the following proposition.

\begin{prop} \label{prop3}
If an isosceles triangle is partitioned into an isosceles triangle and an isosceles trapezoid in the Golden ratio, then their centers of mass form a centroidal Voronoi tessellation of the triangle.
\end{prop}

\begin{proof} Let $\tri$ be the isosceles triangle with uniform density $\rho$. It is enough to prove the proposition for the triangle with vertices $O(0, 0)$, $A(1, 0)$, and $B(1,1)$ (see Figure~3). Let $\ell$ be the line which partitions $\tri$ into an isosceles triangle and an isosceles trapezoid. Let $\ell$ intersects the sides $OA$ and $AB$ at the points $C$ and $D$, respectively. Let $AC=AD=\ga$. Then, the area of the triangle $ACD$ is $\frac 12 \ga^2$, and so, the area of the isosceles trapezoid is $\frac 12(1-\ga^2)$. The equation of the line $OB$ is $x_2=x_1$, and the equation of the line $CD$ is $x_2=x_1-1+\ga$. Let $P(u_1, u_2)$ and $Q(v_1, v_2)$ be the centers of mass of the triangle $ACD$ and the isosceles trapezoid, respectively. Then, $(u_1, u_2)=(\frac{3-\ga} 3, \frac{\ga}3)$. Now, using \eqref{eq0}, we have
\begin{align*}
v_1=\frac{   \int_0^{1-\alpha } x_1 \left(x_1-0\right) \, dx_1+   \int_{1-\alpha }^1 x_1 \left(x_1-\left(\alpha +x_1-1\right)\right) \, dx_1}{   \int_0^{1-\alpha } \left(x_1-0\right) \, dx_1+   \int_{1-\alpha }^1 \left(x_1-\left(\alpha +x_1-1\right)\right) \, dx_1}=\frac{-\alpha ^2+2 \alpha +2}{3 \alpha +3},
\end{align*}
and
\[v_2=\frac{   \int_0^{1-\alpha } \frac{1}{2} \left(x_1^2-0^2\right) \, dx_1+   \int_{1-\alpha }^1 \frac{1}{2} \left(x_1^2-\left(\alpha +x_1-1\right){}^2\right) \, dx_1}{   \int_0^{1-\alpha } \left(x_1-0\right) \, dx_1+   \int_{1-\alpha }^1 \left(x_1-\left(\alpha +x_1-1\right)\right) \, dx_1}=\frac{\alpha ^2+\alpha +1}{3 \alpha +3}.\]
If $P$ and $Q$ form a Centroidal Voronoi tessellation, then the line $CD$ must be the perpendicular bisection of the line segment joining $P$ and $Q$. Thus, we have \[\frac 12(u_2+v_2)=\frac 12(u_1+v_1)-1+\ga.\]
Next, putting the values of $u_1, u_2, v_1$ and $v_2$, and solving the equation, we have $\ga=\frac{1}{2} (\sqrt{5}-1)$, which is the Golden ratio. Since $\ga^2+\ga=1$, we see that
\[\frac{\te{Area of the isosceles triangle } BCD}{ \te{Area of the isosceles trapezoid } OCDB }=\frac {\frac 12 \ga^2}{\frac 12(1-\ga^2)}=\ga.\]
Thus, the proof of the proposition is complete.
\end{proof}

Let $\F R$ be the region occupied by a lamina with uniform density such that the boundary of $\F R$ forms a rhombus. Then, it can be seen that the center of mass of the lamina is located at the center of the rhombus, i.e., at the point where the two diagonals intersect, in other words, the optimal set of one-mean with respect to the uniform distribution defined on the region $\F R$ consists of the center of the rhombus. The following proposition gives the optimal sets of two-means and the corresponding quantization error with respect to the uniform distribution defined on such a region $\F R$.

\begin{prop}\label{prop4}.
Let $\F R$ be the region bounded by the rhombus with vertices $O(0, 0)$, $A(1, 0)$, $B(1+\frac 1 {\sqrt 2}, \frac 1{\sqrt{2}})$ and $C(\frac 1{ \sqrt 2}, \frac 1 {\sqrt 2})$. Then, with respect to the uniform distribution the optimal set of two-means is the set $\set{(\frac{1}{3} (1+\frac{1}{\sqrt{2}}),\frac{1}{3 \sqrt{2}}), (\frac{1}{3} (\sqrt{2}+2),\frac{\sqrt{2}}{3})}$, and the corresponding quantization error is $V_2=\frac{2 \sqrt{2}-1}{18 \sqrt{2}}=0.0718274$.
\end{prop}
\begin{proof} Let $OABC$ be the rhombus with vertices $O(0, 0)$, $A(1, 0)$, $B(1+\frac 1 {\sqrt 2}, \frac 1{\sqrt{2}})$ and $C(\frac 1{ \sqrt 2}, \frac 1 {\sqrt 2})$. Let $f(x_1, x_2)$ be the probability density function of the random variable $X:=(X_1, X_2)$ taking values on the region $\F R$ bounded by the rhombus. Then, $f(x_1, x_2)=\sqrt 2$ for $(x_1, x_2) \in \F R$, and zero, otherwise. Let the position vectors of $A, \, B, \, C$ be $\tilde a, \, \tilde b,\, \tilde c$, respectively. Then, $\tilde a=(1, 0), \, \tilde b=(1+\frac 1{\sqrt 2},  \frac 1{\sqrt 2})$, and $\tilde c=(\frac 1{\sqrt 2}, \frac 1{\sqrt 2})$. Let $P$ and $Q$ be the locations of the two points in an optimal set of two-means. Let $\tilde p$ and $\tilde q$ be the position vectors of $P$ and $Q$, respectively.
Let $\ell$ be the boundary of the Voronoi regions of $P$ and $Q$. Then, the following cases can arise:

\begin{figure}
\begin{tikzpicture}[line cap=round,line join=round,>=triangle 45,x=1.0cm,y=1.0cm]
\draw [color=cqcqcq,, xstep=2.0cm,ystep=2.0cm] (-0.610876523952786,-0.748866757612501) grid (7.422994115258084,3.546052422119194);
\clip(-0.610876523952786,-0.748866757612501) rectangle (7.422994115258084,3.546052422119194);
\draw (0.,0.)-- (4.,0.);
\draw (4.,0.)-- (6.82842712474619,2.82842712474619);
\draw (0.,0.)-- (2.82842712474619,2.82842712474619);
\draw (2.82842712474619,2.82842712474619)-- (6.82842712474619,2.82842712474619);
\draw [line width=1.2pt,dotted,color=ffqqqq] (2.2999173553718992,2.2999173553718992)-- (2.7338842975206603,0.);
\draw (-0.5622568022052936,0.2333942415617488) node[anchor=north west] {$O$};
\draw [line width=1.2pt,dotted] (-0.22479338842975222,0.21305785123966936)-- (2.0892561983471065,2.527107438016527);
\draw [line width=1.2pt,dotted] (-0.009917355371901252,-0.2993388429752078)-- (2.7338842975206603,-0.2993388429752078);
\draw [line width=1.2pt,dotted] (2.7338842975206603,0.)-- (2.7338842975206603,-0.2993388429752078);
\draw [line width=1.2pt,dotted] (2.0892561983471065,2.527107438016527)-- (2.2999173553718992,2.2999173553718992);
\draw [line width=1.2pt,dotted] (0.,0.)-- (-0.22479338842975222,0.21305785123966936);
\draw [line width=1.2pt,dotted] (0.,0.)-- (-0.009917355371901252,-0.2993388429752078);
\draw (0.49354252538383114,1.8208289033870073) node[anchor=north west] {$\beta$};
\draw (1.008192906301418,-0.2578007374776609) node[anchor=north west] {$\alpha$};
\draw (3.959037729083406,0.20255345760400983) node[anchor=north west] {$A$};
\draw (6.802198643285615,3.17308646767143) node[anchor=north west] {$B$};
\draw (2.558606218653839,3.439104349170005) node[anchor=north west] {$C$};
\draw[color=ffqqqq]  (2.944387348165819,0.179898711936393) node {$D$};
\draw[color=ffqqqq] (2.558606218653839,2.3085283528010616) node {$E$};
\end{tikzpicture}
\caption{}\label{Fig4}
\end{figure}

Case~1. $\ell$ intersects the sides $OA$ and $OC$.

Let $\ell$ intersect $OA$ and $OC$ at the points $D$ and $E$, such that the lengths of $OD$ and $OE$ be $\ga$ and $\gb$, respectively, with their position vectors $\tilde d$ and $\tilde e$ (see Figure~\ref{Fig4}). Then, $\tilde d=\alpha  \tilde a$, $\tilde e=\beta  \tilde c$, $\tilde p=\frac{\tilde d+\tilde e}{3}$, and
\[\tilde q=\frac{\frac{1}{\sqrt{2}}\cdot \frac 12 (1+\frac{1}{\sqrt{2}}, \frac{1}{\sqrt{2}})-\frac{\alpha  \beta }{2 \sqrt{2}} \tilde p}{\frac{1}{\sqrt{2}} (1-\frac{\alpha  \beta }{2})},\]
where the location of the center of mass of the lamina is $\frac 12 (1+\frac{1}{\sqrt{2}}, \frac{1}{\sqrt{2}})$.
Since $\ell$ is the boundary of the Voronoi regions of $P$ and $Q$, it is the perpendicular bisector of the line segment joining the points $P$ and $Q$. Thus, we have \[\rho(\tilde p, \tilde d)-\rho(\tilde q, \tilde d)=0, \te{ and } \rho(\tilde p, \tilde e)-\rho(\tilde q, \tilde e)=0.\] Solving the above two equations, we have $\ga=1$ and $\gb=1$, which implies the fact that the line $\ell$ is the diagonal $AC$ of the rhombus. Then, we have $\tilde p=(\frac{1}{3} (1+\frac{1}{\sqrt{2}}),\frac{1}{3 \sqrt{2}})$, $\tilde q=(\frac{2}{3} (1+\frac{1}{\sqrt{2}}),\frac{\sqrt{2}}{3})$. Notice that the equation of the line $OC$ is $x_2=x_1$, and the equation of the line $AC$ is $x_2=-(\sqrt 2+1)(x_1-1)$. Thus, if $V_2(\te{Case }1)$ is the distortion error in this case, then, due to symmetry of the rhombus with respect to the diagonal $AC$, we have
\begin{align*}
V_2(\te{Case }1)=2 \sqrt 2  \int _0^{\frac{1}{\sqrt{2}}}\int _{x_2}^{1-\frac{x_2}{\sqrt{2}+1}}\rho((x_1, x_2), \tilde p) dx_1dx_2= \frac{2 \sqrt{2}-1}{18 \sqrt{2}}=0.0718274.
\end{align*}

Case~2. $\ell$ intersects the sides $AB$ and $BC$.

This case is the reflection of Case~1 with respect to the diagonal $AC$, and thus, we obtain the same set of solutions and the corresponding distortion error as in Case~1.

Case~3. $\ell$ intersects the sides $OA$ and $AB$.

Let $\ell$ intersect $OA$ and $AB$ at the points $D$ and $E$ such that the lengths of $AD$ and $AE$ be $\ga$ and $\gb$, respectively, with their position vectors $\tilde d$ and $\tilde e$. Then, $\tilde d=(1-\alpha)  \tilde a$, $\tilde e=\beta  \tilde b+(1-\gb)\tilde a$, $\tilde p=\frac{1}{3} (\tilde a+\tilde d+\tilde e)$, and
\[\tilde q=\frac{\frac{1}{\sqrt{2}}\cdot \frac 12 (1+\frac{1}{\sqrt{2}}, \frac{1}{\sqrt{2}})-\frac{\alpha  \beta }{2 \sqrt{2}} \tilde p}{\frac{1}{\sqrt{2}} (1-\frac{\alpha  \beta }{2})},\]
where the location of the center of mass of the lamina is $\frac 12 (1+\frac{1}{\sqrt{2}}, \frac{1}{\sqrt{2}})$. As Case~1, we have
\[\rho(\tilde p, \tilde d)-\rho(\tilde q, \tilde d)=0, \te{ and } \rho(\tilde p, \tilde e)-\rho(\tilde q, \tilde e)=0.\]
Solving the two equations, we have $\ga=1$ and $\gb=1$, which implies the fact that the line $\ell$ is the diagonal $OB$ of the rhombus. Then, we have $\tilde p=(\frac{1}{3} (2+\frac{1}{\sqrt{2}}),\frac{1}{3 \sqrt{2}})$, $\tilde q=(\frac{1}{3} \left(\sqrt{2}+1\right),\frac{\sqrt{2}}{3})$. Notice that the equation of the line $AB$ is $x_2=x_1-1$, and the equation of the line $OB$ is $x_2=(\sqrt 2-1)x_1$. Thus, if $V_2(\te{Case }3)$ is the distortion error in this case, then, due to symmetry of the rhombus with respect to the diagonal $OB$, we have
\begin{align*}
V_2(\te{Case }3)=2 \sqrt 2 \int _0^{\frac{1}{\sqrt{2}}}\int _{(\sqrt{2}+1) x_2}^{x_2+1}\rho((x_1, x_2), \tilde p) dx_1dx_2=\frac{2 \sqrt{2}+1}{18 \sqrt{2}}=0.150395.
\end{align*}

\begin{figure}
\begin{tikzpicture}[line cap=round,line join=round,>=triangle 45,x=1.0cm,y=1.0cm]
\draw [color=cqcqcq,, xstep=2.0cm,ystep=2.0cm] (-0.610876523952786,-1.071034914404049) grid (7.222994115258084,4.1238842653276455);
\clip(-0.610876523952786,-1.071034914404049) rectangle (7.222994115258084,4.1238842653276455);
\draw (0.,0.)-- (4.,0.);
\draw (4.,0.)-- (6.82842712474619,2.82842712474619);
\draw (0.,0.)-- (2.82842712474619,2.82842712474619);
\draw (2.82842712474619,2.82842712474619)-- (6.82842712474619,2.82842712474619);
\draw (-0.4822568022052936,0.3133942415617488) node[anchor=north west] {$O$};
\draw (0.765858616804052,-0.3299688942692095) node[anchor=north west] {$\alpha$};
\draw (3.959037729083406,0.20255345760400983) node[anchor=north west] {$A$};
\draw (6.752198643285615,3.1308646767143) node[anchor=north west] {$B$};
\draw (2.2862901272336185,3.1725998787953606) node[anchor=north west] {$C$};
\draw [color=ffqqqq](5.79639931569649,2.8835730565474983) node[anchor=north west] {$E$};
\draw [color=ffqqqq](1.618964936774825,0.534589220601159) node[anchor=north west] {$D$};
\draw [line width=1.2pt,dotted] (0.,0.)-- (0.,-0.3738186189762358);
\draw [line width=1.2pt,dotted] (2.,0.)-- (2.0139740358133977,-0.3516504621846886);
\draw [line width=1.2pt,dotted,color=ffqqqq] (5.985250369024952,2.82842712474619)-- (2.,0.);
\draw [line width=1.2pt,dotted] (0.,-0.3738186189762358)-- (2.0139740358133977,-0.3516504621846886);
\draw (4.12442374432011,3.793794857834771) node[anchor=north west] {$\beta$};
\draw [line width=1.2pt,dotted] (2.7855362948373568,3.128750154088334)-- (2.8309223100740604,3.1509183108798813);
\draw [line width=1.2pt,dotted] (2.8309223100740604,3.1509183108798813)-- (5.985250369024952,3.1509183108798813);
\draw [line width=1.2pt,dotted] (5.985250369024952,3.1509183108798813)-- (5.985250369024953,2.8284271247461903);
\draw [line width=1.2pt,dotted] (2.8284271247461903,2.8284271247461903)-- (2.8309223100740306,3.1509183108798817);
\end{tikzpicture}
\caption{}\label{Fig5}
\end{figure}

Case~4. $\ell$ intersects the sides $OC$ and $BC$.

This case is the reflection of Case~3 with respect to the diagonal $OB$, and thus, we obtain the same set of solutions and the corresponding distortion error as in Case~3.

Case~5. $\ell$ intersects the two opposite sides $OA$ and $BC$.

Let $\ell$ intersect the sides $OA$ and $BC$ at the points $D$ and $E$, such that the lengths of $OD$ and $CE$ be $\ga$ and $\gb$, respectively, with their position vectors $\tilde d$ and $\tilde e$ (see Figure~\ref{Fig5}). Then, $\tilde d=\ga \tilde a$, $\tilde e=\beta  \tilde b+(1-\gb)\tilde c$,
\[\tilde p=\frac {\frac{\alpha }{2 \sqrt{2}}\frac{\tilde c+\tilde d}{3} +\frac{\beta }{2 \sqrt{2}}\frac{1}{3} (\tilde c+\tilde d+\tilde e)}{\frac{\alpha }{2 \sqrt{2}}+\frac{\beta }{2 \sqrt{2}}}, \te{ and } \tilde q=\frac{\frac{1-\alpha }{2 \sqrt{2}} \frac{1}{3} (\tilde a+ \tilde b+\tilde d)+\frac{1-\beta }{2 \sqrt{2}} \frac{1}{3} (\tilde b+\tilde d+\tilde e)}{\frac{1-\alpha }{2 \sqrt{2}}+\frac{1-\beta }{2 \sqrt{2}}}.\]
As Case~1, we have
\[\rho(\tilde p, \tilde d)-\rho(\tilde q, \tilde d)=0, \te{ and } \rho(\tilde p, \tilde e)-\rho(\tilde q, \tilde e)=0.\]
Solving the above equations, we obtain two sets of solutions: $\set{\ga=1, \, \gb=0}$, and $\set{\ga=0, \, \gb=1}$. If  $\set{\ga=1, \, \gb=0}$, then the results obtained in this case are same as the results obtained in Case~1. If $\set{\ga=0, \, \gb=1}$, then the results obtained in this case are same as the results obtained in Case~3.

Case~6. $\ell$ intersects the two opposite sides $AB$ and $OC$.

Let $\ell$ intersect the sides $AB$ and $OC$ at the points $D$ and $E$, such that the lengths of $AD$ and $OE$ be $\ga$ and $\gb$, respectively, with their position vectors $\tilde d$ and $\tilde e$. Then, $\tilde d= (1-\alpha )\tilde a+\alpha  \tilde b$, $\tilde e=\beta  \tilde c$,
\[\tilde p=\frac {\frac{\beta }{2 \sqrt{2}} \frac{\tilde a+\tilde e}{3}+ \frac{\alpha }{2 \sqrt{2}} \frac{1}{3} (\tilde a+\tilde d+\tilde e)}{\frac{\alpha }{2 \sqrt{2}}+\frac{\beta }{2 \sqrt{2}}}, \te{ and } \tilde q=\frac{\frac{1-\alpha }{2 \sqrt{2}} \frac{1}{3} (\tilde b+\tilde c+\tilde d)+\frac{1-\beta }{2 \sqrt{2}}\frac{1}{3} (\tilde c+\tilde d+\tilde e)  }{\frac{1-\alpha }{2 \sqrt{2}}+\frac{1-\beta }{2 \sqrt{2}}}.\]
As Case~1, we have
\[\rho(\tilde p, \tilde d)-\rho(\tilde q, \tilde d)=0, \te{ and } \rho(\tilde p, \tilde e)-\rho(\tilde q, \tilde e)=0.\]
Solving the above equations, we obtain two sets of solutions: $\set{\ga=1, \, \gb=0}$, and $\set{\ga=0, \, \gb=1}$. Thus, we see that the results obtained in this case are same as the results obtained in Case~5.

Recall that optimal set of two-means gives the smallest distortion error. Thus, considering all the above possible cases, we see that the set $\set{(\frac{1}{3} (1+\frac{1}{\sqrt{2}}),\frac{1}{3 \sqrt{2}}), \, (\frac{2}{3} (1+\frac{1}{\sqrt{2}}),\frac{\sqrt{2}}{3})} $ forms a unique optimal set of two-means with quantization error $V_2=\frac{2 \sqrt{2}-1}{18 \sqrt{2}}=0.0718274$. Thus, the proof of the proposition is complete.
\end{proof}

\begin{remark} \label{rem1}
From the proof of Proposition~\ref{prop4}, we see that the region bounded by a rhombus has two different centroidal Voronoi tessellations (CVTs) with two generators, and the CVT with the smallest distortion error gives the optimal set of two-means with respect to the uniform distribution.
\end{remark}
In the next section, we describe the optimal quantization for some discrete uniform distributions.

\section{optimal quantization for discrete unform distributions}
Let $S:=\set{(x_i, y_j): 1\leq i,  j\leq n}$ for some positive integer $n$. Then, $S$ is a data set containing $n^2$ observations. Let $\tilde {X}:=(X, Y)$ be a random vector taking values on $S$ with a discrete uniform distribution $P$. $P$ being discrete uniform, the mass function $f(x, y)$ of $P$ is given by
\[f(x, y)=\left\{\begin{array}{ccc}
\frac 1 {n^2} & \te{ for } (x, y) \in S, \\
\ 0  & \te{ otherwise}.
\end{array}\right.
\]
Let $E(\tilde {X})$ be the expected vector of $\tilde {X}$. Then, we have
\begin{align} \label{eq34}
E(\tilde {X})&=\sum_{u, v=1}^n (x_u i+y_v j)f(x_u, y_v)=\Big(\frac 1 n (x_1+x_2+\cdots +x_n), \frac 1 n (y_1+y_2+\cdots +y_n)\Big),
\end{align}
i.e., $E(\tilde {X})=(E(X), E(Y))$, where $E(X)=\frac 1 n (x_1+x_2+\cdots +x_n)$, and $E(Y)= \frac 1 n (y_1+y_2+\cdots +y_n)$, implying the fact that the expected vector of the random vector $\tilde {X}$ with respect to the discrete uniform distribution is the mean of the data set $S$. Proceeding in the similar way we can prove the following proposition:
 \begin{prop} \label{prop34} Let $P$ be a discrete uniform distribution on a data set $S$ containing finite number of observations.  Let $A\sci S$. Then, the conditional expected vector $E(\tilde {X} : \tilde {X} \in A)$ of the random vector $\tilde {X}$ taking values on $A$ with distribution $P$ is the mean of the data points belonging to the subset $A$.
 \end{prop}
 Recall that if $\ga$ is an optimal set of $n$-means for a probability distribution $P$, then for any $a\in \ga$, $a$ is the expected value (vector) of its own Voronoi region. Using the above proposition we can determine the optimal sets of $n$-means and the $n$th quantization errors for many finite discrete distributions as illustrated in the following examples. In the following two examples, we give the optimal sets $\ga_n$ of $n$-means, and the $n$th quantization errors $V_n(P)$ for some $n$ for two different discrete distributions. Notice that in the two examples the elements in the two data sets are symmetrically located, and the associated probability distributions are also uniform. It is not difficult to determine the optimal sets of $n$-means for smaller values of $n\leq \te{card(supp(}P))$ for such data sets with uniform distributions. Thus, here we do not show the details of the calculations. Notice that in the figures the optimal quantizers are denoted by `$\times$', and the elements in the corresponding Voronoi regions are denoted by the same color.

\begin{example} \label{example1} Consider the data set $S$ given by
\[S:=\set{(0, 0), (\frac 1 3, 0), (\frac 23, 0), (1, 0), (\frac 1 6, \frac {\sqrt 3}{6}), (\frac 13, \frac {\sqrt 3} 3), (\frac 12, \frac {\sqrt 3}{2}), (\frac 56, \frac{\sqrt 3}{6}), (\frac 23, \frac {\sqrt 3}{3})}\] associated with the probability mass function $f(x, y)$ given by $f(x, y)=\frac 19 $ if $(x, y)\in S$, and zero, otherwise.
Then, by \eqref{eq34}, we have $E(\tilde {X})=(\frac 12, \frac {\sqrt 3}{6})$, i.e., the optimal set of one-mean is the singleton $\set{(\frac 12, \frac {\sqrt 3}{6})}$ and the corresponding quantization error is the variance $V$ given by
\[V=V_1(P)=\sum_{(x, y) \in S}\|(x, y)-(\frac 12, \frac {\sqrt 3}{6})\|^2 f(x, y)=0.185185.\]

Now, due to Proposition~\ref{prop34}, after some calculations, we have

$(i)$ $\ga_2=\set{(0.5, 0.673575), (0.5, 0.096225)}$ with quantization error $0.111111$. Notice that due to rotational symmetry there are three different optimal sets of three-means (see Figure~\ref{Fig22}).
\begin{figure}
\centerline{\includegraphics[width=3.5 in, height=3.5 in, keepaspectratio] {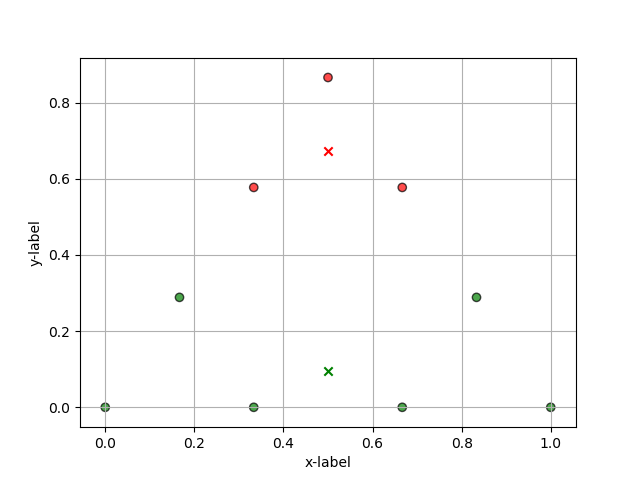}\includegraphics[width=3.5 in, height=3.5 in, keepaspectratio] {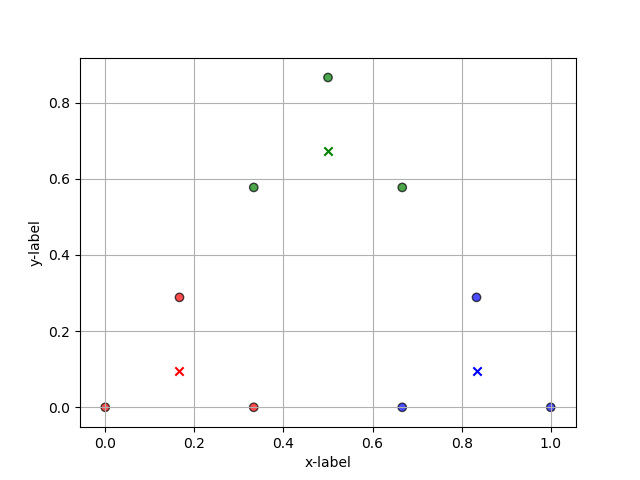}}
\caption{Optimal sets of $n$-means for $n=2$ and $n=3$.} \label{Fig22}
\end{figure}

$(ii)$ $\ga_3=\set{(0.166666,  0.096225), (0.833333,  0.096225),  (0.5,  0.673575)}$ with quantization error $0.037037$. Notice that the optimal set of three-means is unique (see Figure~\ref{Fig22}).

\begin{figure}
\centerline{\includegraphics[width=3.5 in, height=3.5 in, keepaspectratio] {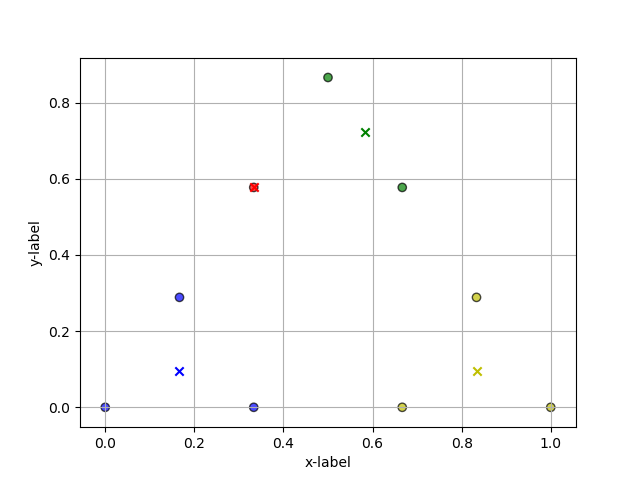}
\includegraphics[width=3.5 in, height=3.5 in, keepaspectratio] {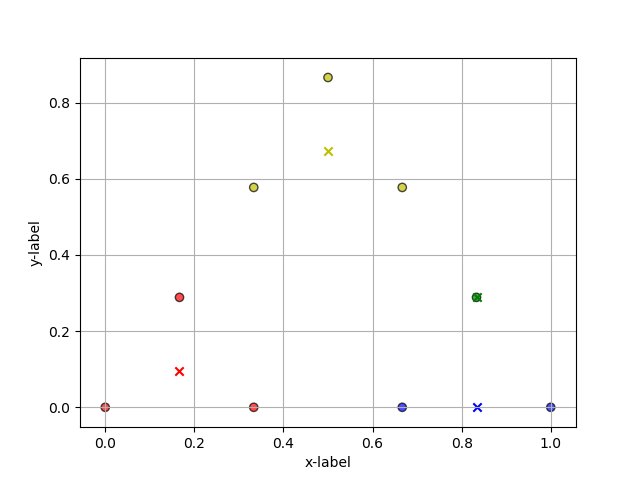}}
\centerline{\includegraphics[width=3.5 in, height=3.5 in, keepaspectratio] {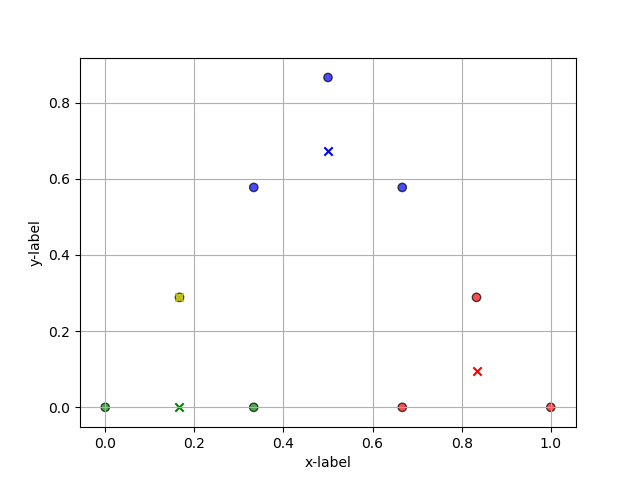}
\includegraphics[width=3.5 in, height=3.5 in, keepaspectratio] {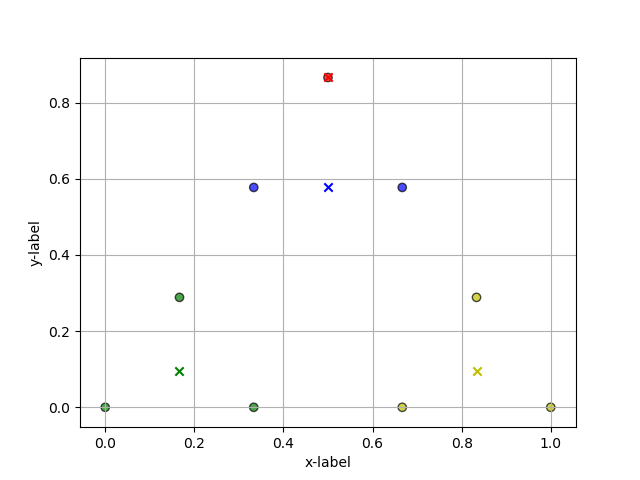}}
\caption{Optimal sets of $n$-means for $n=4$.} \label{Fig24}
\end{figure}

\begin{figure}
\centerline{\includegraphics[width=3.5 in, height=3.5 in, keepaspectratio] {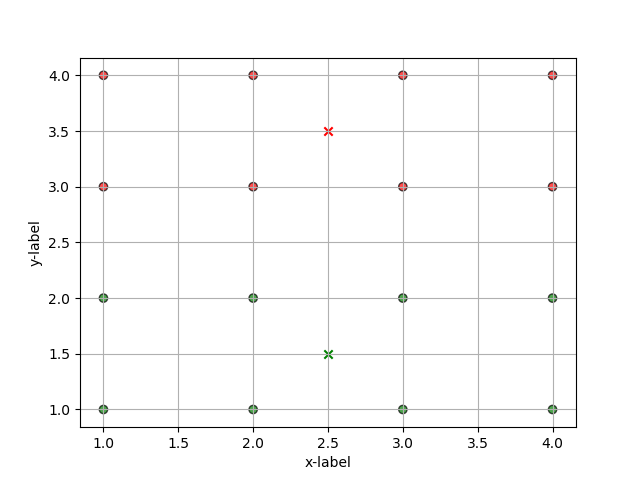} \includegraphics[width=3.5 in, height=3.5 in, keepaspectratio] {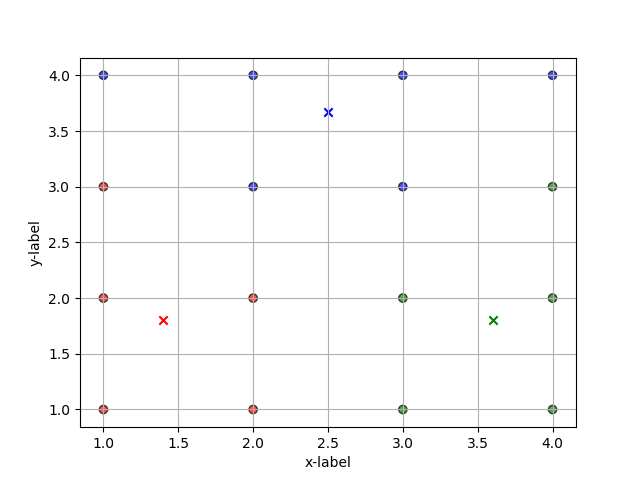}}
\caption{Optimal sets of $n$-means for $n=2$ and $n=3$.} \label{Fig32}
\end{figure}

\begin{figure}
\centerline{\includegraphics[width=3.5 in, height=3.5 in, keepaspectratio] {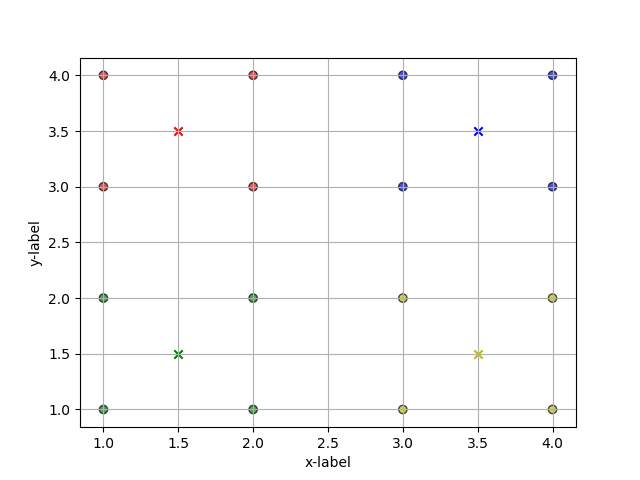}}
\caption{Optimal set of $n$-means for $n=4$.} \label{Fig34}
\end{figure}

$(iii)$ $\ga_4=\set{(0.833333,  0.096225), (0.166666,  0), (0.5, 0.673575), (0.166666,  0.288675)}$ with quantization error  $0.030864$. There are several optimal sets of four-means with the same quantization error $0.030864$ (see Figure~\ref{Fig24}).
\end{example}

\begin{figure}
\centerline{\includegraphics[width=3.5 in, height=3.5 in, keepaspectratio] {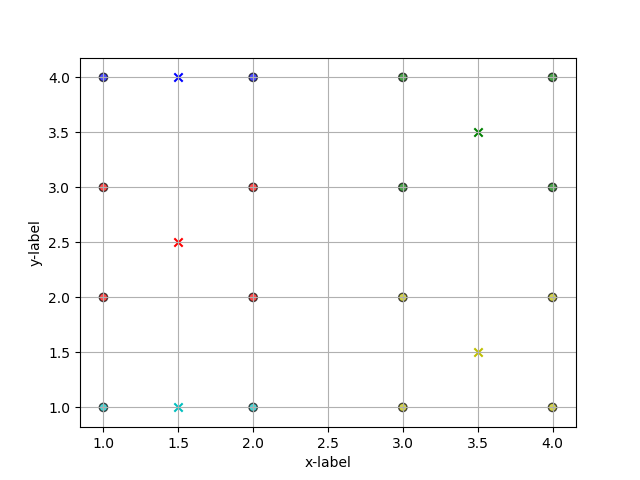} \includegraphics[width=3.5 in, height=3.5 in, keepaspectratio] {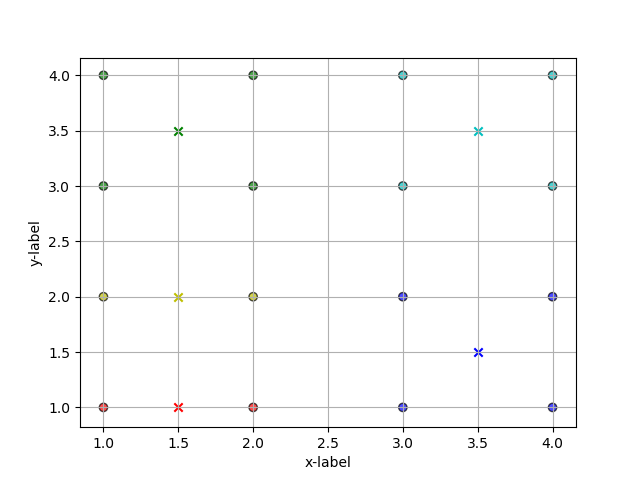}}
\centerline{\includegraphics[width=3.5 in, height=3.5 in, keepaspectratio] {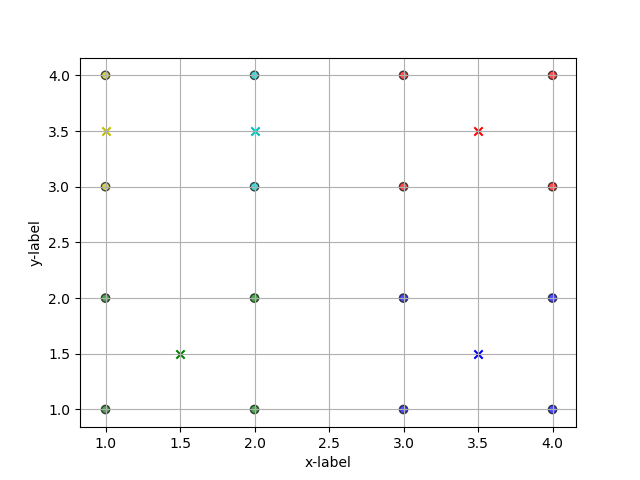} \includegraphics[width=3.5 in, height=3.5 in, keepaspectratio] {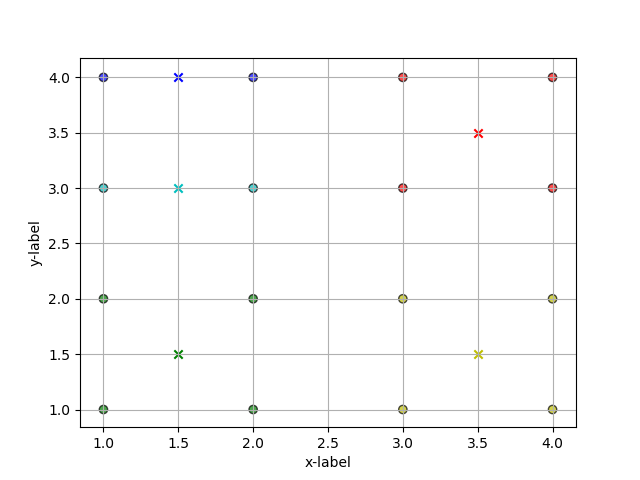}}
\caption{Optimal sets of $n$-means for $n=5$.} \label{Fig35}
\end{figure}

\begin{example} \label{example2} Let $S$ be the data set given by
\[S:=\set{(i, j)\in \D N^2 : 1\leq i, j\leq 4},\]
and the associated probability mass function $f(x_1, x_2)$ is given
\[f(x_1, x_2)=\left\{\begin{array}{cc}
\frac 1{16} & \te{if } (x_1, x_1)\in S,\\
0 & \te{otherwise}.
\end{array}
\right.
\]
Let $X:=(X_1, X_2)$ be the random vector given by the mass function $f(x_1, x_2)$. Then, we have
\[E(X)=\sum_{(x_1,x_2)\in \D R^2}(x_1 i+x_2 j) f(x_1, x_2)=(\frac 5 2, \frac 52),\]
i.e., the optimal set of one-mean is the singleton $\set{(\frac 52, \frac 52)}$ and the corresponding quantization error is the variance $V$ given by
\[V=V_1(P)=\sum_{(x_1, x_2) \in S}\|(x_1, x_2)-(\frac 52, \frac 52)\|^2 f(x_1, x_2)=2.5.\]
Now, due to Proposition~\ref{prop34}, after some calculations, we have

$(i)$ $\ga_2=\set{(2.5, 1.5), (2.5, 3.5)}$ with quantization error $1.5$. Notice that due to rotational symmetry there are two different optimal sets of two-means (see Figure~\ref{Fig32}).

$(ii)$ $\ga_3=\set{(1.4,  1.8), (3.6,  1.8),  (2.5,  3.666667)}$ with quantization error $0.927083$. Notice that due to rotational symmetry there are four different optimal sets of three-means (see Figure~\ref{Fig32}).

$(iii)$ $\ga_4=\set{(1.5,  3.5), (1.5,  1.5),  (3.5,  3.5), (3.5,  1.5)}$ with quantization error $0.5$. Notice that the optimal set of four-means is unique (see Figure~\ref{Fig34}).

$(iv)$ $\ga_5=\set{(3.5,  3.5), (1.5,  1.5),  (1.5,  4), (3.5,  1.5), (1.5,  3)}$ with quantization error $0.4375$. Notice that there are several optimal sets of five-means (see Figure~\ref{Fig35}).
\end{example}
We now conclude the paper with the following remark.

\begin{figure}
\centerline{\includegraphics[width=3.5 in, height=3.5 in, keepaspectratio] {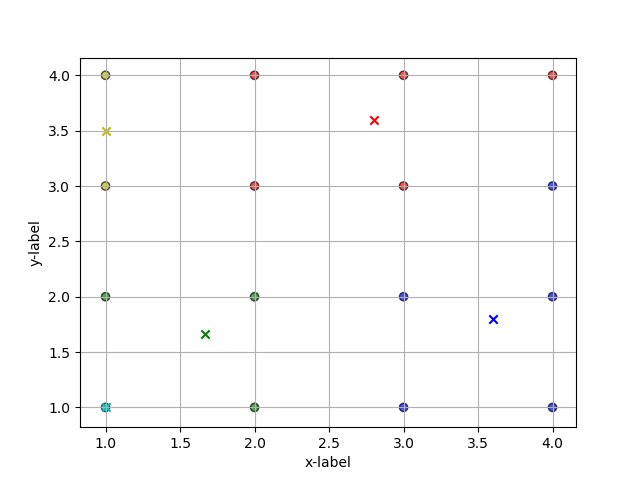}}
\caption{The set of the means does not form an optimal set.} \label{Fig3666}
\end{figure}

\begin{remark}
From Example~\ref{example1} and Example~\ref{example2}, we see that if $\ga$ is an optimal set of $n$-means, then each $a\in \ga$ is the mean of its own Voronoi region. But, the converse is not true. For example, the elements in the set $\gb:=\set{(2.8,  3.6), (1.67,  1.67), (3.6,  1.8), (1,  3.5), (1,1)}$ are the means of their own Voronoi regions (see Figure~\ref{Fig3666}) for the data set associated with the uniform distribution given by Example~\ref{example2}. But, the set $\gb$ is not an optimal set of $n$-means for $n=5$, because the distortion error given by the set $\gb$ is $0.614583$ which is larger than the distortion error given by the set $\ga_5=\set{(3.5,  3.5), (1.5,  1.5),  (1.5,  4), (3.5,  1.5), (1.5,  3)}$ as described in $(iv)$ of Example~\ref{example2}.
\end{remark}

\end{document}